\documentclass{amsart}

\usepackage[utf8]{inputenc}
\usepackage[T1]{fontenc} 
\usepackage{csquotes}
\DeclareUnicodeCharacter{14D}{{\=o}}

\usepackage{amsmath,mathrsfs,amssymb}

\usepackage{enumitem}
\usepackage[english]{babel}
\usepackage{hyperref}
\usepackage[all,cmtip]{xy}
\entrymodifiers={+!!<0pt,\fontdimen22\textfont2>}

\usepackage[a4paper,centering]{geometry}
\usepackage[parfill]{parskip} 

\newtheorem{thm}{Theorem}[section]
\newtheorem{lem}[thm]{Lemma}
\newtheorem{prop}[thm]{Proposition}
\newtheorem{cor}[thm]{Corollary}

\theoremstyle{definition}
\newtheorem{defn}[thm]{Definition}
\newtheorem{ex}[thm]{Example}

\theoremstyle{remark}
\newtheorem{rem}[thm]{Remark}
\usepackage{enumitem}

\def\C{{\mathbb C}} 
\def\N{{\mathbb N}}

\def\B{{B}} 
\def\Cpt{{K}} 
\def\H{{\mathcal H}}
\def\K{{\mathcal K}}
\def\M{{M}}

\mathchardef\ordinarycolon\mathcode`\: 
\def\vcentcolon{\mathrel{\mathop\ordinarycolon}} 
\providecommand*\coloneqq{\mathrel{\vcentcolon\mkern-1.2mu}=}

\def\cast{$C^{*}$}

\providecommand{\id}{\mathop{\rm id}\nolimits}

\def\Cl{C^{*}_{\lambda}}
\def\Cu{C^{*}_{u}}

\def\R{\mathbb R}

\def\cb{\mathrm{cb}}
\newcommand\bk[1]{{\langle #1 \rangle}}

\newcommand\rrtimes{\rtimes_{r}}

\title{The Fubini Product and Its Applications}
\date{\today}   

\author{Otgonbayar Uuye and Joachim Zacharias}

\address[Otgonbayar Uuye]{Institute of Mathematics\\ National University of Mongolia\\Ulaanbaatar\\Mongolia}
\address[Otgonbayar Uuye]{Department of Mathematics\\School of Arts and Sciences\\ National University of Mongolia\\Ih Surguuliin Gudamj 1\\Ulaanbaatar\\Mongolia}
\email{otogo@num.edu.mn}

\address[Joachim Zacharias]{School of Mathematics and Statistics\\University of Glasgow\\15 University Gardens\\ G128QW\\UK}

 \email{Joachim.Zacharias@glasgow.ac.uk}
   

\thanks{This research is partially funded by the Advanced Research Grant UTS-2015-SHUS-14 of the National University of Mongolia (OU) and by the Engineering and Physical Sciences Research Council Grant EP/I019227/2 (JZ)}

\begin{document}
\maketitle

\begin{abstract} 
The Fubini product of operator spaces provide a powerful tool for analysing properties of tensor products. In this paper we review the the theory of Fubini products and apply it to the problem of computing invariant parts of dynamical systems. In particular, we study the invariant translation approximation  property of discrete groups. 
\end{abstract}

\section{Introduction}

The Fubini product of \cast-algebras was first defined and studied by J. Tomiyama \cite{MR0218906,MR0397427,MR0450985,MR0420296} and Wassermann \cite{MR0410402}. It is now a standard tool in the study of operator algebras and operator spaces. See for instance \cite{MR567831,MR641529,MR715549,MR985277,MR1138840,MR1721796,MR2391387}. Unfortunately, to our knowledge, a comprehensive treatment of the subject is missing from the literature.

In this paper, we try to the remedy the situation. In Section \ref{section Fubini}, we give a survey of the theory of Fubini products and its applications. Most of the results in Section \ref{section Fubini} appear scattered in numerous articles, most notably \cite{MR0397427,MR0410402,MR0328609,MR567831,MR1138840,MR1220905,MR3272046}. After briefly recalling some definitions concerning operator spaces in Subsection \ref{sub operator spaces}, we define the Fubini product and prove its fundamental properties regarding functoriality, intersections, kernels, relative commutants, invariant elements, combinations in Subsections \ref{sub fubini product}-\ref{sub combinations}. In Subsection \ref{sub OAP}, we review the relation between the operator approximation property and the slice map property. 


In Section \ref{section ITAP}, we apply the results in Section \ref{section Fubini} to the study of groups with the invariant translation approximation property (ITAP) of J. Roe \cite[Section 11.5.3]{MR2007488}. In Subsection \ref{sub roe algebra}, we recall the definition of the uniform Roe algebra. In Subsection \ref{sub products}, we analyse the uniform Roe algebra of a product space. Finally, in Subsection \ref{sub itap}, we study the ITAP of product groups. We show that for countable discrete groups $G$ and $H$, if $G$ has the approximation property (AP) of Haagerup-Kraus \cite{MR1220905}, the product $G \times H$ has the ITAP if and only if $H$ has the ITAP. 

Finally, in Section \ref{section crossed}, we study the crossed product version of the Fubini product.

\section{The Fubini Product}\label{section Fubini}

\subsection{Operator Spaces}\label{sub operator spaces}
For the sake of completeness, we start with some definitions concerning operator spaces. See \cite{MR1793753,MR1976867} for a complete treatment.

Let $\H$ be a Hilbert space and let $\B(\H)$ denote the Banach space of bounded linear operators on $\H$. The ideal of compact operators on $\H$ is denoted $\Cpt(\H) \subseteq \B(\H)$. 
\begin{defn} An {\em operator space} on $\H$ is a closed subspace of $\B(\H)$. 
\end{defn}

\begin{defn} For $a \in \B(\H)$ and $b \in \B(\K)$, we write $a \otimes b$ for the corresponding element in $\B(\H \otimes \K)$. For subsets $A \subseteq \B(\H)$ and $B \subseteq \B(\K)$, we write $A \times B$ for the subset $\{a \otimes b \mid a \in A, b \in B\} \subseteq \B(\H \otimes \K)$ and $A \odot B$ for the linear span of $A \times B$.
\end{defn}

\begin{defn} Let $A \subseteq \B(\H)$ and $B \subseteq \B(\K)$ be operator spaces. We define the {\em tensor product} $A \otimes B$ as the norm closure of $A \odot B$ or, equivalently, the closed linear span of  $A \times B$ in $\B(\H \otimes \K)$. 
\end{defn}

\begin{rem} Suppose $A \subseteq \B(\H_{1})$, $A \subseteq \B(\H_{2})$, $B \subseteq \B(\K_{1})$ and $B \subseteq \B(\K_{2})$. Then for any $x \in A \odot B$, we have
	\begin{equation*}
	\|x\|_{\B(\H_{1} \otimes \K_{1})} = \|x\|_{\B(\H_{2} \otimes \K_{2})}.
	\end{equation*}
\end{rem}

\begin{lem}[({\cite[Lemma 3]{MR0218906}})]\label{lem tomiyama} Suppose $S$, $A \subseteq \B(\H)$ and $T$, $B \subseteq \B(\K)$. If $S \otimes T \subseteq A \otimes B$, then $S \subseteq A$ and $T \subseteq B$. \qed
\end{lem}

\begin{defn} A linear map $\phi\colon A \to B$ of operator spaces is {\em completely bounded} if the map $\phi \odot \id_{D} \colon A \odot D \to B \odot D$ is bounded for all operator spaces $D$. We write $\phi \otimes \id_{D}$ for the extension $A \otimes D \to B \otimes D$.

Completely contractive and completely isometric maps are defined similarly.
\end{defn}

\begin{defn} The {\em completely bounded norm} of a completely bounded map $\phi \colon A \to B$ is
	\begin{equation*}
	\|\phi\|_{\cb} \coloneqq \left\|\phi \otimes \id_{\Cpt(\H)} \colon A \otimes \Cpt(\H) \to B \otimes \Cpt(\H)\right\|,
	\end{equation*}
where $\H$ is a separable infinite dimensional Hilbert space.	
\end{defn}

\begin{lem} Let $\phi\colon A \to B$ be a completely bounded map. Then we have
	\begin{equation*}
	\|\phi\|_{\cb} = \sup_{n} \left\|\phi \otimes \id_{\M_{n}} \colon A \otimes M_{n} \to B \otimes M_{n}\right\|.
	\end{equation*}
Moreover, for any operator space $D$, we have
	\begin{equation*}
	\|\phi \otimes \id_{D}\| \le \|\phi\|_{\cb}.
	\end{equation*}		
\qed	
\end{lem}


The following extension theorem of Wittstock says that $\B(\H)$ is an injective object in the category of operator spaces and completely contractive maps.
\begin{thm}[(Wittstock)] Let $S \subseteq A$ be operator spaces and let $\phi\colon S \to \B(\H)$ be a completely bounded map. Then there exists a completely bounded extension $\psi\colon A \to \B(\H)$ with $\|\psi\|_{\cb} = \|\phi\|_{\cb}$.
\qed
\end{thm}

\subsection{The Fubini Product}\label{sub fubini product}
\begin{defn}[({cf.\ \cite{MR0397427}, \cite{MR0410402}})] Let $S \subseteq A$ and $T \subseteq B$ be operator spaces. The {\em Fubini product} $F(S, T, A \otimes B)$ of $S$ and $T$ 
is defined as the set of all $x \in A \otimes B$ such that $(\phi \otimes \id_{B})(x) \in T$ for all $\phi \in A^{*}$ and $(\id_{A} \otimes \psi)(x) \in S$ for all $\psi \in B^{*}$. 

\end{defn}

\begin{rem} 
We  have 
	\begin{equation*}
	S \otimes T \subseteq F(S, T, A \otimes B) \subseteq A \otimes B.
	\end{equation*}
\end{rem}
\begin{lem}
Let $K, L$ be subspaces of $A^*, B^*$ respectively such that  the closed unit ball of $K$ and $L$ are weak-$*$-dense in the unit balls of $A^*$ and $B^*$ respectively. Then for any $S \subseteq A, T \subseteq B$ operator spaces the Fubini product  $F(S, T, A \otimes B)$ equals the set
\begin{equation*}
\{ x \in A \otimes B \mid (\phi \otimes \id_B )(x) \in T \textup{ and } (\id_A \otimes \psi) (x) \in S \textup{ for all } \phi \in K, \psi \in L \}.
\end{equation*}
\end{lem}
The assumptions are for instance satisfied if $K$ and $L$ are the set of normal linear functionals in faithful representations of $A$ and $B$. 
\begin{proof}
Suppose that $ (\phi \otimes \id_B )(x) \in T $ for all $\phi \in L$. We need to show that $ (\phi \otimes \id_B )(x) \in T $ for all $\phi \in A^*$. Let $\phi \in A^*$ and let $(\phi_n) \subseteq L$ be a bounded sequence (or net) converging pointwise (i.e. in the weak-$*$-sense) to $\phi$. Then $ (\phi_n \otimes \id_B )(z) \to   (\phi \otimes \id_B )(z) $ for all $ z\in  A \odot B$ and using norm boundedness of $ \phi_n \otimes \id_B $, an $\varepsilon/3$-argument shows that the same holds for $z \in A \otimes B$. Since $ (\phi_n \otimes \id_B )(x) \in T $ by assumption we conclude that $ (\phi \otimes \id_B )(x) \in T$. 
\end{proof}

\begin{defn}
We say that $(S, T, A \otimes B)$ has the {\em slice map property} if 
	\begin{equation*}
	F(S, T, A \otimes B) = S \otimes T.
	\end{equation*}
\end{defn}

\subsection{Functoriality}

\begin{lem} For $i = 1, 2$, let $S_{i} \subseteq A_{i}$ and $T_{i} \subseteq B_{i}$ be operator spaces. Suppose that completely bounded maps $\sigma\colon A_{1} \to A_{2}$ and $ \tau \colon B_{1} \to B_{2}$ satisfy $\sigma(S_{1}) \subseteq S_{2}$ and $\tau(T_{1}) \subseteq T_{2}$. Then
	\begin{equation*}
	(\sigma \otimes \tau)(F(S_{1}, T_{1}, A_{1} \otimes B_{1})) \subseteq F(S_{2}, T_{2}, A_{2} \otimes B_{2}).
	\end{equation*}
\end{lem}
\begin{proof} Let $x \in F(S_{1}, T_{1}, A_{1} \otimes B_{1})$. Let $\phi_{2} \in A_{2}^{*}$ and let $\phi_{1} \coloneqq \phi_{2} \circ \sigma \in A_{1}^{*}$. Then $(\phi_{1} \otimes \id_{B_{1}})(x) \in T_{1}$, thus
	\begin{align*}
	(\phi_{2} \otimes \id_{B_{2}})[(\sigma \otimes \tau) (x)] 
	&= \tau [(\phi_{1} \otimes \id_{B_{1}})(x)]
	\end{align*}
belongs to $T_{2}$. Similarly for $\psi_{2} \in B_{2}^{*}$, we have $(\id_{A_{2}} \otimes \psi_{2})[(\sigma \otimes \tau)(x)] \in S_{2}$. Thus $(\sigma \otimes \tau)(x) \in F(S_{2}, T_{2}, A_{2} \otimes B_{2})$.
\end{proof}

\begin{lem} Let $S \subseteq A$ and $T \subseteq B$ be operator spaces. Let $\sigma \colon A \to A$ and $\tau\colon B \to B$ be completely bounded maps. If $\sigma$ restricts to the identity on $S$ and $\tau$ on $T$, then $\sigma \otimes \tau$ restricts to the identity on $F(S, T, A \otimes B)$.
\end{lem}
\begin{proof} For $x \in F(S, T, A \otimes B)$ and $\phi \in A^{*}$ and $\psi \in B^{*}$, we have
	\begin{align*}
	\bk{\phi \otimes \psi, (\sigma \otimes \tau)(x)} &= \bk{\phi, \sigma[(\id_{A} \otimes (\psi \circ\tau))(x)]}\\
	&=\bk{\phi, (\id_{A} \otimes (\psi \circ\tau))(x)}\\
	&= \bk{\psi, \tau[(\phi \otimes \id_{B})(x)]}\\
	&= \bk{\phi \otimes \psi, x}.
	\end{align*}
Thus $(\sigma \otimes \tau)(x) = x$.	
\end{proof}

\begin{cor}[({cf.\ \cite[Lemma 2]{MR567831}})]
For $i = 1, 2$, let $S \subseteq A_{i}$ and $T \subseteq B_{i}$ be operator spaces. Suppose that completely bounded maps $\sigma_{1}\colon A_{1} \to A_{2}$, $\sigma_{2}\colon A_{2} \to A_{1}$ and $\tau_{1} \colon B_{1} \to B_{2}$, $\tau_{2} \colon B_{2} \to B_{1}$ satisfy, for $i = 1, 2$, $\sigma_{i}(s) = s$ for $s \in S$ and $\tau_{i}(t) = t$ for $t \in T$, then $\sigma_{1} \otimes \tau_{1}$ restricts to an isomorphism 
	\begin{equation*}
	F(S, T, A_{1} \otimes B_{1}) \cong F(S, T, A_{2} \otimes B_{2}).
	\end{equation*}
\qed	
\end{cor}

\begin{cor}[({cf.\ \cite[Proposition 3.7]{MR0397427}})] Let $S \subseteq A$ and $T \subseteq B$ be operator spaces. If there exist completely bounded  projections $A \to S$ and $B \to T$, then $(S, T, A \otimes B)$ has the slice map property:
	\begin{equation*}
	F(S, T, A \otimes B) = S \otimes T. 
	\end{equation*}
\qed
\end{cor}
\subsection{Intersections}
\begin{lem}\label{lem int} For operator spaces $S \subseteq A$ and $T \subseteq B$, we have
	\begin{equation*}
	F(S, T, A \otimes B) = F(A, T, A \otimes B) \cap F(S, B, A \otimes B).	
	\end{equation*}
More generally, for families of operator spaces $\{S_{\alpha} \subseteq A\}$ and $\{T_{\beta} \subseteq B\}$, we have
	\begin{equation*}
	F(\cap_{\alpha} S_{\alpha}, \cap_{\beta} T_{\beta}, A \otimes B) = \cap_{\alpha, \beta} F(S_{\alpha}, T_{\beta}, A \otimes B).
	\end{equation*}	
\end{lem}
\begin{proof} Clear from the definitions.
\end{proof}

\begin{cor}[({cf.\ \cite[Corollary 5]{MR0410402}})] Let $S_{1}$, $S_{2} \subseteq A$ and $T_{1}$, $T_{2} \subseteq B$ be operator spaces. If $(S_{1} \cap S_{2}, T_{1} \cap T_{2}, A \otimes B)$ has the slice map property, then
	\begin{equation*}
	(S_{1} \cap S_{2}) \otimes (T_{1} \cap T_{2}) = (S_{1} \otimes T_{1}) \cap (S_{2} \otimes T_{2}).
	\end{equation*}	
\end{cor}
\begin{proof} We have 
	\begin{align*}
	(S_{1} \cap S_{2}) \otimes (T_{1} \cap T_{2}) &\subseteq (S_{1} \otimes T_{1}) \cap (S_{2} \otimes T_{2})\\
	&\subseteq F(S_{1}, T_{1}, A \otimes B) \cap F(S_{2}, T_{2}, A \otimes B)\\
	&=  F(S_{1} \cap S_{2}, T_{1} \cap T_{2}, A \otimes B).
	\end{align*}
\end{proof}

\subsection{Kernels}
\begin{prop}\label{prop ker} Let $A$, $B$ and $D$ be operator spaces and let $\tau\colon B \to D$ be a completely bounded map. Then
	\begin{equation*}
	F(A, \ker(\tau), A \otimes B) = \ker (\id_{A} \otimes \tau).
	\end{equation*}
\end{prop}
\begin{proof} Let $x \in A \otimes B$. For $\phi \in A^{*}$, we have 
	\begin{equation*}
	\tau [(\phi \otimes \id_{B})(x)] = (\phi \otimes \id_{D})[(\id_{A} \otimes \tau)(x)].
	\end{equation*}
Thus $x \in F(A, \ker(\tau), A \otimes B)$ iff $x \in \ker (\id_{A} \otimes \tau)$.	
\end{proof}

\begin{cor}[({cf.\ \cite[Corollary 1]{MR0397427}})] Let $A$, $B$ and $D$ be operator spaces and let $\tau \colon B \to D$ be a completely bounded map. Then the equality
	\begin{equation*}
	A \otimes \ker(\tau) = \ker (\id_{A} \otimes \tau)
	\end{equation*}
holds if and only if the triple $(A, \ker(\tau), A \otimes B)$ has the slice map property.\qed	
\end{cor}

\begin{thm}\label{thm ker} Let $A$, $B$, $C$ and $D$ be operator spaces and let $\{\sigma_{\alpha}\colon A \to C\}$ and $\{\tau_{\beta}\colon B \to D\}$ be families of completely bounded maps. Then
	\begin{equation*}
	F(\cap_{\alpha}\ker(\sigma_{\alpha}), \cap_{\beta}\ker(\tau_{\beta}), A \otimes B) = (\cap_{\alpha}\ker(\sigma_{\alpha} \otimes \id_{B})) \cap (\cap_{\beta}\ker (\id_{A} \otimes \tau_{\beta})).
	\end{equation*}
\end{thm}
\begin{proof} Follows from Lemma~\ref{lem int} and Proposition~\ref{prop ker}. 
\end{proof}

\subsection{Relative Commutants}
As a corollary we obtain the following.
\begin{prop}\label{prop com} Let $S$, $A \subseteq \B(\H)$ and $T$, $B \subseteq \B(\K)$ be operator spaces. Then we have
	\begin{equation*}
	F(S' \cap A, T' \cap B, A \otimes B) = (S \otimes \C1_{\B(\K)} + \C1_{\B(\H)} \otimes T)' \cap (A \otimes B).
	\end{equation*}
\end{prop}
\begin{proof} For $s \in S$,  let $\sigma_{s} \coloneqq [-, s] \colon A \to \B(\H)$. Then $\sigma_{s}$ is completely bounded and
	\begin{equation*}
	\sigma_{s} \otimes \id_{B} = [-, s \otimes 1_{\B(\K)}] \colon A \otimes B \to \B(\H) \otimes  B.
	\end{equation*}
Moreover, we have
	\begin{align*}
	\cap_{s \in S} \ker(\sigma_{s}) &= S' \cap A \quad\text{and}\\
	\cap_{s \in S} \ker(\sigma_{s} \otimes \id_{B}) &= (S \otimes \C1_{\B(\K)})' \cap (A \otimes B).
	\end{align*}	
Similary, $\tau_{t} \coloneqq [-, t] \colon B \to \B(\K)$, $t\in T$,  are  completely bounded and
	\begin{align*}
	\cap_{t \in T} \ker(\tau_{t}) &= T' \cap B \quad\text{and}\\
	\cap_{t \in T} \ker(\id_{A} \otimes \tau_{t}) &= (\C 1_{\B(\H)} \otimes T)' \cap (A \otimes B).
	\end{align*}		 
Thus Theorem~\ref{thm ker} completes the proof.
\end{proof}

\begin{cor}[({cf.\ \cite[Theorem 1]{MR0328609} and \cite[Corollary 2]{MR0397427}})]\label{cor com} Let $S$, $A \subseteq \B(\H)$ and $T$, $B \subseteq \B(\K)$ and suppose $1_{\B(\H)} \in S$ and $1_{\B(\K)} \in T$. Then the equality 
	\begin{equation*}
	(S' \cap A) \otimes (T' \cap B) = (S \otimes T)' \cap (A \otimes B)
	\end{equation*}
holds if and only if the triple $(S' \cap A, T' \cap B, A \otimes B)$ has the slice map property. 	
\end{cor}
\begin{proof} Follows from Proposition~\ref{prop com}, since under the unitality conditions, we have 
	\begin{equation*}
	(S \otimes T)' = (S \otimes \C 1_{\B(\K)} + \C 1_{\B(\H)} \otimes T)'.
	\end{equation*}
\end{proof}
\begin{cor}[({cf.\ \cite[Corollary 1]{MR0328609}})] Let $A$ and $B$ be \cast-algebras. Then
	\begin{equation*}
	Z(A) \otimes Z(B) \cong Z(A \otimes B),
	\end{equation*}
where $Z$ denotes the center.	
\end{cor}
\begin{proof} Follows from Corollary~\ref{cor com} and Remark~\ref{rem CPAP to OAP}.
\end{proof}

\subsection{Invariant Elements}
The following can be deduced from Proposition~\ref{prop com}. We give a direct proof.
\begin{prop}\label{prop inv} Let $A$ and $B$ be operator spaces. Suppose that the group $G$ acts on $A$ and the group $H$ acts on $B$ by completely bounded maps. Then
	\begin{equation*}
	F(A^{G}, B^{H}, A \otimes B) = (A \otimes B)^{G \times H}. 
	\end{equation*}
\end{prop}
\begin{proof} For $g \in G$, let $\sigma_{g} \coloneqq (\id_{A} - g)\colon A \to A$. Then $\sigma_{g}$ is completely bounded and 
	\begin{equation*}
	\sigma_{g} \otimes \id_{B} = (\id_{A \otimes B} - 1_{H} \times g) \colon A \otimes B \to A \otimes B.
	\end{equation*}
Moreover, we have
	\begin{align*}
	\cap_{g \in G}\ker(\sigma_{g}) &= A^{G}\\
%
	\cap_{g \in G}\ker(\sigma_{g} \otimes \id_{B}) &= (A \otimes B)^{G \times \{1_{H}\}}.
	\end{align*}
Similarly, $\tau_{h} \coloneqq (\id_{B} - h) \colon B \to B$, $h \in H$, are completely bounded and
	\begin{align*}
	\cap_{h \in H}\ker(\tau_{h}) &= B^{H}\\
%
	\cap_{h \in H}\ker(\id_{A} \otimes \tau_{h} ) &= (A \otimes B)^{\{1_{G}\} \times H}.
	\end{align*}
Now Theorem~\ref{thm ker} completes the proof, since 
	\begin{equation*}
	(A \otimes B)^{G \times H} = (A \otimes B)^{G \times \{1_{H}\}} \cap (A \otimes B)^{\{1_{G}\} \times H}.
	\end{equation*}
\end{proof}

\begin{cor}[({cf.\ \cite[Corollary 7]{MR0410402}})]\label{cor inv} Let $A$ and $B$ be operator spaces. Suppose that group $G$ acts on $A$ and group $H$ acts on $B$ by completely bounded maps. Then the equality  
	\begin{equation*}
	A^{G} \otimes B^{H} = (A \otimes B)^{G \times H}
	\end{equation*}
holds if and only if the triple $(A^{G}, B^{H}, A \otimes B)$ has the slice map property.
\end{cor}

\subsection{Combinations}\label{sub combinations}

\begin{lem}\label{lem fstab} Let $S \subseteq A$ and $T \subseteq B$ be operator spaces. Then
	\begin{equation*}
	F(S, T, A \otimes T) = (A \otimes T) \cap F(S, B, A \otimes B).
	\end{equation*}
\end{lem}
\begin{proof} Clear since any element of $T^{*}$ extends to an element of $B^{*}$.
\end{proof}
\begin{lem}[({\cite[Lemma 2.6]{MR3272046}})]\label{lem aotimes} Let $S \subseteq A$ and $T \subseteq B$ be operator spaces. If $(A, T, A \otimes B)$ has the slice map property, then
	\begin{equation*}
	F(S, T, A \otimes T) = F(S, T, A \otimes B)
	\end{equation*}
\end{lem}
\begin{proof} We have
	\begin{align*}
	F(S, T, A \otimes T) &= (A \otimes T) \cap F(S, B, A \otimes B)\\
	&= F(A, T, A \otimes B) \cap F(S, B, A \otimes B)\\
	& = F(S, T, A \otimes B),
	\end{align*}
by Lemma~\ref{lem fstab}	and Lemma~\ref{lem int}.
\end{proof}

\begin{prop}[({\cite[Corollary 2.7]{MR3272046}})]\label{prop trans} Let $S \subseteq A$ and $T \subseteq B$ be operator spaces. If $(A, T, A \otimes B)$ and $(S, T, A \otimes T)$ have the slice map property, then $(S, T, S \otimes B)$ and $(S, T, A \otimes B)$ also have the slice map property. Conversely, if $(S, T, A \otimes B)$ has the slice map property, then $(S, T, A \otimes T)$ and $(S, T, S \otimes B)$ also have the slice map property. 
\end{prop}
\begin{proof} Follows from Lemma~\ref{lem aotimes} and the following commutative diagram of inclusions:
	\begin{equation*}
	\xymatrix{
	F(S, T, A \otimes T) \ar@{^{(}->}[r] & F(S, T, A \otimes B)\\
	S \otimes T \ar@{^{(}->}[r] \ar@{^{(}->}[u]& F(S, T, S \otimes B)  \ar@{^{(}->}[u]\\
 	}
	\end{equation*}
\end{proof}

\subsection{The Operator Approximation Property}\label{sub OAP}

In this subsection, we review the connection between the operator approximation property and the slice map property. This is partly for completeness and partly because we use it in Section~\ref{section ITAP}. However, we have nothing new to add here to the excellent work of Kraus \cite{MR1138840} and Haagerup-Kraus \cite{MR1220905}.   

\begin{defn}
We say that $A$ has the {\em slice map property} for $B$ if $(A, T, A \otimes B)$ has the slice map property for all operator spaces $T \subseteq B$.
\end{defn}

\begin{defn} We write $F(A, B)$ for the space of finite-rank maps $A \to B$.
\end{defn}

\begin{lem} Finite-rank maps of operator spaces are completely bounded. \qed
\end{lem}

Let $A$ and $B$ be operator spaces and let $x \in A \otimes B$. Define 
	\begin{align*}
	F_{B}(x) &\coloneqq \{(\Phi \otimes \id_{B})(x) \mid \Phi \in F(A, A)\}^{=} \subseteq A \otimes B, \\
	T_{B}(x) &\coloneqq \{(\phi \otimes \id_{B})(x) \mid \phi \in A^{*}\}^{=} \subseteq B.
	\end{align*} 
Then 
	\begin{equation*}
	F_{B}(x) = A \otimes T_{B}(x).
	\end{equation*}
Moreover, we have
	\begin{align*}
	F(A, T, A \otimes B) &= \{x \in A \otimes B \mid T_{B}(x) \subseteq T\}\\
	&= \{x \in A \otimes B \mid F_{B}(x) \subseteq A \otimes T\}.
	\end{align*}

\begin{lem}[({\cite[Theorem 5.4]{MR1138840}})]\label{lem main kraus}
Let $A$ and $B$ be operator spaces. Then $A$ has the slice map property for $B$ if and only if $x \in F_{B}(x)$ for all $x \in A \otimes B$.
\end{lem}
\begin{proof} ($\Rightarrow$): Let $x \in A \otimes B$. Clearly, $x \in F(A, T_{B}(x), A \otimes B)$. Since $A$ has the slice map property for $B$, we have $F(A, T_{B}(x), A \otimes B) = A \otimes T_{B}(x) = F_{B}(x)$. Thus $x \in F_{B}(x)$.

($\Leftarrow$): Let $T \subseteq B$ be operator subspace and let $x \in F(A, T, A \otimes B)$. Then $F_{B}(x) \subseteq A \otimes T$. Thus $x \in A \otimes T$.
\end{proof}

\begin{defn} An operator space $B$ is {\em matrix stable} if for each $n \in \N$, there is a completely bounded surjection $B \to B \otimes M_{n}$.
\end{defn}

\begin{defn} We say that  $A$ has the {\em operator approximation property} (OAP) for $B$ if there is a net $\Phi_{\alpha} \in F(A, A)$  such that $\Phi_{\alpha} \otimes \id_{B}$ converges to $\id_{A} \otimes \id_{B}$ in point-norm topology.
\end{defn}
\begin{thm}[({cf.\ \cite[Theorem 5.4]{MR1138840}})]\label{thm kraus} Let $A$ and $B$ be operator spaces. If $A$ has the operator approximation property for $B$, then $A$ has the slice map property for $B$. If $B$ is matrix stable, then the converse also holds.

\end{thm}


\begin{proof} The first statement is clear from Lemma~\ref{lem main kraus}.

Now we prove the second statement. Suppose $B$ is stable. Let $x_{1}, \dotsc, x_{n} \in A \otimes B$ and let $\epsilon > 0$. Let $C \coloneqq B \otimes M_{n}$ and choose a completely bounded surjection $\pi\colon B \to C$.

Let $x \coloneqq x_{1} \oplus \dotsb \oplus x_{n} \in A \otimes B \otimes M_{n} = A \otimes C$. Then there exists $y \in A \otimes B$ such that $x = (\id_{A} \otimes \pi)(y)$. By Lemma~\ref{lem main kraus}, there is $\Phi \in F(A, A)$ such that $\|(\Phi \otimes \id_{B})(y) - y\| < \epsilon/(\|\pi\|_{\cb} + 1)$. Then 
	\begin{align*}
	\|(\Phi \otimes \id_{C})(x) - x\| &= 	\|(\Phi \otimes \id_{C})((\id_{A} \otimes \pi)(y)) - (\id_{A} \otimes \pi)(y)\|\\
	&=\|(\id_{A} \otimes \pi) ((\Phi \otimes \id_{B}) (y) -y)\|\\ 
	&< \epsilon.
	\end{align*}
It follows, for any $1 \le k \le n$, we have $\|(\Phi \otimes \id_{B})(x_{k}) - x_{k}\| < \epsilon$.	
\end{proof}

\begin{defn} Let $\H$ is a separable infinite dimensional Hilbert space. We say that $A$ has the {\em operator approximation property} (OAP) if $A$ has the OAP for $\Cpt(\H)$ and the {\em strong operator approximation property} (SOAP) if $A$ has the OAP for $\B(\H)$.
\end{defn}

\begin{lem} If an operator space has the strong OAP, then it has the OAP for any $B$. \qed
\end{lem}

\begin{rem}\label{rem CPAP to OAP} We have the implications
	\begin{center}
	CPAP $\implies$ CBAP $\implies$ strong OAP $\implies$ OAP.
	\end{center}
\end{rem}

\begin{defn}[({\cite[Definition 1.1]{MR1220905}})] A countable discrete group $G$ has the {\em approximation property} (AP) if the constant function $1$ is in the $\sigma(M_{0}A(G), Q(G))$-closure of $A(G)$ in $M_{0}A(G)$, where $A(G)$ is the Fourier algebra of $G$ and $M_{0}A(G)$ is the space of completely bounded Fourier multipliers of $G$ and $Q(G)$ is the standard predual of $M_{0}A(G)$.  
\end{defn}

\begin{thm}[({\cite[Theorem 2.1]{MR1220905}})]\label{thm HK} A countable discrete group $G$ has the AP if and only if its reduced group \cast-algebra $\Cl(G)$ has the (strong) OAP.
\end{thm}
\begin{proof} See the original article \cite{MR1220905} or \cite[Section 12.4]{MR2391387}. 
\end{proof}

\section{The Invariant Translation Approximation Property}\label{section ITAP}

In this section, we apply the results in Section~\ref{section Fubini} to the problem of studying the invariant part of the uniform Roe algebra.

\subsection{Uniform Roe Algebras}\label{sub roe algebra}
\begin{defn} We say that a (countable discrete) metric space is of {\em bounded geometry} if for any $R > 0$, there is $N_{R} < \infty$ such that all balls of radius at most $R$ have at most $N_{R}$ elements. 
\end{defn}

\begin{defn}[({cf.\ \cite[Section 4.1]{MR2007488} and \cite[Section 2]{MR2215118}})]
Let $(X, d)$ be a metric space of bounded geometry and let $S \subseteq \B(\H)$ be a subset. For $R > 0$ and $M > 0$, let $A_{R, M}(X, S)$ denote the set of $X \times X$-matrices with values in $S$ satisfying
	\begin{enumerate}
	\item for any $x_{1}$, $x_{2 }\in X$ with $d(x_{1}, x_{2}) > R$, we have $a(x_{1}, x_{2}) = 0$ and
	\item for any $x_{1}$, $x_{2 }\in X$, we have $\|a(x_{1}, x_{2})\|_{\B(\H)} \le M$.
	\end{enumerate}
Let $A(X, S) \coloneqq \cup_{R, M} A_{R, M}(X, S)$.	For $S = \C$, we write $A_{R, M}(X)$ and $A(X)$.
\end{defn}

\begin{lem}\label{lem norm bound} Under the natural action, elements of $A_{R, M}(X, S)$ act on $l^{2}(X, \H) \cong l^{2}X \otimes \H$ as bounded operators of norm at most $M \cdot N_{R}$. 
\end{lem}
\begin{proof} The action is given by
	\begin{equation*}
	(a \xi)(x_{1}) \coloneqq \sum_{x_{2}} a(x_{1}, x_{2}) \xi(x_{2}), \quad \xi \in l^{2}X \otimes \H. 
	\end{equation*}
Thus \begin{align*}
	\|(a \xi)(x_{1})\| &\le \sum_{d(x_{1}, x_{2})\le R} \|a(x_{1}, x_{2}) \xi(x_{2})\|\\
	&\le M \sum_{d(x_{1}, x_{2}) \le R} \|\xi(x_{2})\|\\
	&\le M \left(N_{R} \sum_{d(x_{1}, x_{2}) \le R} \|\xi(x_{2})\|^{2}\right)^{1/2}.
	\end{align*}
It follows that
	\begin{align*}
	\|a \xi\|^{2} &= \sum_{x_{1}} \|(a\xi)(x_{1})\|^{2}\\ 
	&\le M^{2} N_{R}  \sum_{x_{1}}\sum_{d(x_{1}, x_{2}) \le R} \|\xi(x_{2})\|^{2}\\
	&\le M^{2}N_{R}^{2} \sum_{x_{2}} \|\xi(x_{2})\|^{2}\\
	&= M^{2} N_{R}^{2} \|\xi\|^{2}.
	\end{align*}	
\end{proof}

Thus we identify $A(X, S) \subseteq \B(l^{2}X \otimes \H)$.

\begin{lem}\label{lem A inclusions} For any $S \subseteq \B(\H)$, we have
	\begin{equation*}
	A(X) \times S \subseteq A(X, S) \subseteq A(X, \overline S) \subseteq \overline{A(X, S)} 
	\end{equation*}
in $\B(l^{2}X \otimes \H)$.	
\end{lem}
\begin{proof} Only the last inclusion needs checking. Take $a \in A_{R, M}(X, \overline S)$. For each positive integer $n \ge 1$, we define $a_{n} \in A(X, S)$ as follows. For $x_{1}$, $x_{2} \in X$, if $d(x_{1}, x_{2}) > R$, then let $a_{n}(x_{1}, x_{2}) = 0$, and if $d(x_{1}, x_{2}) \le R$, then choose $a_{n}(x_{1}, x_{2}) \in S$ to satisfy $\|a_{n}(x_{1}, x_{2}) - a(x_{1}, x_{2})\| \le 1/n$. Then $a_{n} \in A_{R, M + 1}(X, S)$ and $\|a_{n} - a\| \le N_{R}/n$ by Lemma~\ref{lem norm bound}. Hence the sequence $a_{n}$ converges to $a$ in $\B(l^{2}X \otimes \H)$ as $n \to \infty$ and thus $a \in \overline{A(X, S)}$.
\end{proof}

\begin{defn}[({cf.\ \cite[Section 4.4]{MR2007488} \cite{MR2215118}})] Let $X$ be a metric space of bounded geometry and let $S \subseteq \B(H)$ be an operator space. The {\em uniform Roe operator space} $\Cu(X, S)$ is the closure of $A(X, S)$ in $\B(l^{2}X \otimes \H)$. For $S = \C$, we write $\Cu(X)$.
\end{defn}

\begin{lem} For any operator space $S \subseteq \B(\H)$, we have
	\begin{equation*}
	\Cu(X) \otimes S \subseteq \Cu(X, S).
	\end{equation*}
\end{lem}
\begin{proof} Follows from Lemma~\ref{lem A inclusions}.
\end{proof}

\subsection{Products}\label{sub products}

For metric spaces $(X, d_{X})$ and $(Y, d_{Y})$, we equip the product $X \times Y$ with the metric 
	\begin{equation*}
	d_{X \times Y}((x_{1}, y_{1}), (x_{2}, y_{2})) \coloneqq \max\{d_{X}(x_{1}, x_{2}), d_{Y}(y_{1}, y_{2})\}.
	\end{equation*}
If $X$ and $Y$ are of bounded geometry, then so is $X \times Y$.	
	
\begin{lem}\label{lem in prod} Let $X$ and $Y$ be metric spaces of bounded geometry and let $S \subseteq \B(\H)$ be a subset. Then for $R$, $R' > 0$ and $M$, $M' > 0$ we have a natural inclusion
	\begin{equation*}
	A_{R, M}(X) \times A_{R', M'}(Y, S) \subseteq A_{R'', M''}(X \times Y, S),
	\end{equation*}
where $R'' \coloneqq \max\{R, R'\}$ and $M'' \coloneqq M \cdot M'$. In particular, we have
	\begin{equation*}
	A(X) \times A(Y, S) \subseteq A(X \times Y, S)
	\end{equation*}
in $\B(l^{2}X \otimes l^{2}Y \otimes \H)$.	
\end{lem}
\begin{proof} Take $a \in A_{R, M}(X)$ and $b \in A_{R', M'}(Y, S)$ and let $R'' = \max\{R, R'\}$ and $M'' = M \cdot M'$. For $x_{1}$, $x_{2} \in X$ and $y_{1}, y_{2} \in Y$, we define 
	\begin{equation*}
	(a \otimes b)((x_{1}, y_{1}), (x_{2}, y_{2})) \coloneqq a(x_{1}, x_{2}) \otimes b(y_{1}, y_{2}).
	\end{equation*}
\begin{enumerate}
\item For $x_{1}$, $x_{2} \in X$ and $y_{1}, y_{2} \in Y$ with $d((x_{1}, y_{1}), (x_{2}, y_{2})) > R'' = \max\{R, R'\}$, we have either $d(x_{1}, x_{2}) > R'' \ge R$ thus $a(x_{1}, x_{2}) = 0$, or $d(y_{1}, y_{2}) > R'' \ge R'$ thus $b(y_{1}, y_{2}) = 0$, hence $(a \otimes b)((x_{1}, y_{1}), (x_{2}, y_{2})) = a(x_{1}, x_{2}) \otimes b(y_{1}, y_{2}) = 0$ and
\item For $x_{1}$, $x_{2} \in X$ and $y_{1}, y_{2} \in Y$, we have $\|(a \otimes b)((x_{1}, y_{1}), (x_{2}, y_{2}))\| = \|a(x_{1}, x_{2})\| \cdot \|b(y_{1}, y_{2}) \| \le M \cdot M' = M''$.
\end{enumerate}
Hence $a \otimes b$ belongs to $A_{R'', M''}(X \times Y, S)$. The last statement is clear.
\end{proof}

\begin{lem}\label{lem prod in} Let $X$ and $Y$ be metric spaces of bounded geometry and let $S \subseteq \B(\H)$ be a subset. Then for $R > 0$ and $M > 0$, we have a natural inclusion
	\begin{equation*}
	A_{R, M}(X\times Y, S) \subseteq A_{R, M'}(X, A_{R, M}(Y, S)),
	\end{equation*}
where $M' \coloneqq  M \cdot N^{X}_{R}$.	 In particular, we have
	\begin{equation*}
	A(X \times Y, S) \subseteq A(X, A(Y, S))
	\end{equation*}
in $ \B(l^{2}X \otimes l^{2}Y \otimes \H)$.
\end{lem}
\begin{proof} Take $a \in A_{R, M}(X \times Y, S)$. Fix $x_{1}$, $x_{2} \in X$ and consider $b_{x_{1}, x_{2}}(y_{1}, y_{2}) \coloneqq a((x_{1}, y_{1}), (x_{2}, y_{2}))$. First we show that $b_{x_{1}, x_{2}}$ is an element of $A_{R, M}(Y, S)$. 
\begin{enumerate} 
\item For $y_{1}$, $y_{2} \in Y$ with $d(y_{1}, y_{2}) > R$, we have $d((x_{1}, y_{1}), (x_{2}, y_{2})) \ge d(y_{1}, y_{2}) > R$, thus $b_{x_{1}, x_{2}}(y_{1}, y_{2}) = a((x_{1}, y_{1}), (x_{2}, y_{2})) = 0$.
\item For $y_{1}$, $y_{2} \in Y$, we have $\|b_{x_{1}, x_{2}}(y_{1}, y_{2}) \| = \|a((x_{1}, y_{1}), (x_{2}, y_{2}))\| \le M$.
\end{enumerate}
Now we show that $b$ is an element of $A_{R, M'}(X, A_{R, M}(Y, S))$.
\begin{enumerate} 
\item For $x_{1}$, $x_{2} \in X$ with $d(x_{1}, x_{2}) > R$ and for any $y_{1}$, $y_{2} \in Y$, we have $d((x_{1}, y_{1}), (x_{2}, y_{2})) \ge d(x_{1}, x_{2}) > R$, thus 
	\begin{equation*}
	b_{x_{1}, x_{2}}(y_{1}, y_{2}) = a((x_{1}, y_{1}), (x_{2}, y_{2})) = 0
	\end{equation*}
	 i.e.\ $b_{x_{1}, x_{2}} = 0$
\item For $x_{1}$, $x_{2} \in X$, we have $\|b_{x_{1}, x_{2}}\| \le M \cdot N^{X}_{R} = M'$ by Lemma~\ref{lem norm bound}.
\end{enumerate}

\end{proof}

\begin{thm}\label{thm prod} Let $X$ and $Y$ be metric spaces of bounded geometry and let $S \subseteq \B(H)$ be an operator space. Then we have a natural inclusion
	\begin{equation*}
	\Cu(X) \otimes \Cu(Y, S) \subseteq \Cu(X \times Y, S) \subseteq \Cu(X, \Cu(Y, S)).
	\end{equation*}
\end{thm}
\begin{proof} Follows from Lemmas \ref{lem in prod} and \ref{lem prod in}
\end{proof}

\subsection{The Invariant Translation Approximation Property}\label{sub itap}

Let $G$ be a countable discrete group. Let $|G|$ denote the metric space of bounded geometry on $G$ associated to a proper length function. Then $G$ acts on $|G|$ isometrically by right translations.

\begin{ex} Let $G$ be a countable discrete group. Then $A(G)$ is the $*$-algebra generated by $\C[G]$ and $l^{\infty}(G)$, and therefore $\Cu|G| = C^{*}(l^{\infty}(G), \lambda(G))$. It is well-known that $\Cu|G| \cong l^{\infty}(G) \rrtimes G$ and $(\Cu|G|)^{G} = L(G) \cap \Cu|G|$, where $L(G)$ denote the von Neumann algebra generated by $\Cl(G)$. See \cite[Section 5.1]{MR2391387}.
\end{ex}

\begin{defn} Let $S$ be an operator space. We say $G$ has the {\em invariant translation approximation property} (ITAP) for $S$, if we have
	\begin{equation*}
	\Cl(G) \otimes S = \Cu(|G|, S)^{G}. 
	\end{equation*} 
\end{defn}

\begin{defn} We say $G$ has the {\em (strong) invariant translation approximation property} if it has the ITAP for (all operator spaces $S \subseteq \B(\H)$) $\C$. 
\end{defn}

The following theorem connects the strong ITAP to the AP.
\begin{thm}[({\cite{MR2215118}})]\label{thm Z} Let $G$ be a countable discrete group. Then $G$ has the AP if and only if $G$ is exact and has the strong ITAP. \qed
\end{thm}
\begin{proof} See \cite{MR2215118}.
\end{proof}

\begin{thm}\label{thm ITAP subgroup} Any subgroup of a group with ITAP has ITAP.
\end{thm}
\begin{proof}
Let $H$ be a subgroup of $G$ and suppose that $G$ has the ITAP. Using the coset decomposition of $G$ one checks that we have inclusions $\Cl(H) \subseteq \Cl(G)$, $L(H) \subseteq L(G)$ and $\Cu|H| \subseteq \Cu |G|$. Now the multiplier with respect to the indicator function of $H$, being positive definite, induces conditional expectations $E_{L} \colon L(G) \to L(H)$, $E_{u} \colon \Cu|G| \to \Cu |H|$ and $E_{\lambda} \colon \Cl(G) \to \Cl(H)$. The restrictions of $E_{L}$ and $E_{u}$ to $\Cl(G)$ are equal to $E_{\lambda}$. Now let $x \in L(H) \cap \Cu |H|$ then $x \in L(G) \cap \Cu |G| = \Cl(G)$ and $E_{L}(x) =x$ as well as $E_{u}(x) = x$. It follows that $E_{\lambda}(x)=x$ i.e.\ $x \in \Cl(H)$ and thus $H$ also has the ITAP.
\end{proof}

Now we consider products.

\begin{lem} Let $G$ and $H$ be countable discrete groups with proper length functions $l_{G} \colon G \to \R_{\ge 0}$ and $l_{H} \colon G \to \R_{\ge 0}$. Then $l_{G \times H} \colon G \times H \to \R_{\ge 0}$ given by $l_{G \times H}(g, h) = \max\{l_{G}(g), l_{H}(h)\}$ is a proper length function and $|G \times H| = |G| \times |H|$.
\end{lem}
\begin{proof} First we check that $l$ is a length function. Indeed, we have 
	\begin{equation*}
	l_{G \times H}(e, e) = \max\{l_{G}(e), l_{H}(e)\} = 0
	\end{equation*}
and 
	\begin{align*}
	l_{G \times H}(g^{-1}, h^{-1}) &= \max\{l_{G}(g^{-1}), l_{H}(h^{-1})\}\\
	&= \max\{l_{G}(g), l_{H}(h)\}\\
	&= l_{G \times H}(g, h)
	\end{align*}
and finally
	\begin{align*}
	l_{G \times H}(gg', hh') &= \max\{l_{G}(gg'), l_{H}(hh')\}\\
	&\le \max\{l_{G}(g) + l_{G}(g'), l_{H}(h) + l_{H}(h')\}\\
	&\le \max\{l_{G}(g), l_{H}(h)\} + \max\{l_{G}(g'), l_{H}(h')\}\\
	&= l_{G \times H}(g, h) + l_{G \times H}(g', h').
	\end{align*}
Moreover, it is clear that $l$ is proper. Finally, we have
	\begin{align*}
	d_{|G \times H|}((g_{1}, h_{1}), (g_{2}, h_{2})) &= l_{G \times H}(g_{1}g_{2}^{-1}, h_{1}h_{2}^{-1})\\
	&=\max\{l_{G}(g_{1}g_{2}^{-1}), l_{H}(h_{1}h_{2}^{-1})\}\\
	&=\max\{d_{|G|}(g_{1}, g_{2}), d_{|H|}(h_{1}, h_{2})\}\\
	&= d_{|G| \times |H|}((g_{1}, h_{1}), (g_{2}, h_{2})).
	\end{align*} 
\end{proof}

\begin{prop} Let $G$ and $H$ be countable discrete groups. If $G \times H$ has the ITAP, then $G$ and $H$ also have the ITAP and the triple 
	\begin{equation*}
	((\Cu |G|)^{G}, (\Cu |H|)^{H}, \Cu |G| \otimes \Cu |H|)
	\end{equation*}
has the slice map property.
\end{prop}
\begin{proof} By Theorem~\ref{thm prod}, we have 
	\begin{align*}
	\Cl (G \times H) &= \Cl(G) \otimes \Cl(H)\\
	&\subseteq (\Cu |G|)^{G} \otimes (\Cu |H|)^{H}\\
	&\subseteq (\Cu |G| \otimes \Cu |H|)^{G \times H}\\ 
	&\subseteq (\Cu |G \times H|)^{G \times H}\\
	&= \Cl(G \times H).
	\end{align*}
Thus $G$ and $H$ have the ITAP by Lemma~\ref{lem tomiyama} (this is also immediate from Theorem~\ref{thm ITAP subgroup}). The last statement follows from Corollary~\ref{cor inv}.	
\end{proof}

\begin{thm}\label{thm AP coef} Let $G$ be a countable discrete group with the AP and let $B$ be an operator space equipped with an action of a group $H$ by completely bounded maps. Then we have
	\begin{equation*}
	\Cl(G) \otimes B^{H} = \Cu(|G|, B)^{G \times H}.
	\end{equation*}
\end{thm}
\begin{proof} 
Clearly,  we have
	\begin{equation*}
	\Cl(G) \otimes B^{H} \subseteq (\Cu |G|)^{G} \otimes B^{H} \subseteq (\Cu |G| \otimes B)^{G \times H} \subseteq \Cu(|G|, B)^{G \times H}.
	\end{equation*}
Since $G$ has AP, it has the strong ITAP by Theorem~\ref{thm Z}, thus 
	\begin{equation*}
	\Cu(|G|, B)^{G} = \Cl(G) \otimes B.
	\end{equation*}
Moreover, the \cast-algebra $\Cl(G)$ has the strong OAP by Theorem~\ref{thm HK}, thus by Proposition~\ref{prop inv}, we have
	\begin{align*}
	(\Cl(G) \otimes B)^{H} = F(\Cl(G), B^{H}, \Cl(G) \otimes B) = \Cl(G) \otimes B^{H}.
	\end{align*}
It follows that
	\begin{align*}
	\Cu(|G|, B)^{G \times H} &= \left(\Cu(|G|, B)^{G}\right)^{H}\\
	&\subseteq (\Cl(G) \otimes B)^{H}\\
	&=\Cl(G) \otimes B^{H}.
	\end{align*}
This completes the proof.
\end{proof}

%

\begin{thm} Let $G$ and $H$ be countable discrete groups. If $G$ has the AP and $H$ has the ITAP, then $G \times H$ has the ITAP.
\end{thm}
\begin{proof}
By Theorem~\ref{thm prod}, we have
	\begin{equation*}
	\Cu |G \times H| \subseteq \Cu \left(|G|, \Cu |H|\right).
	\end{equation*}
Thus, by Theorem~\ref{thm AP coef}, we have
	\begin{equation*}
	(\Cu |G \times H|)^{G \times H} \subseteq \Cu \left(|G|, \Cu (|H|)\right)^{G \times H} = \Cl(G) \otimes (\Cu |H|)^{H}.
	\end{equation*}
Since $H$ has ITAP, we have $(\Cu |H|)^{H} = \Cl(H)$. Thus, we obtain
	\begin{align*}
	\Cl(G) \otimes \Cl(H) &= \Cl (G \times H)\\
	&\subseteq (\Cu |G \times H|)^{G \times H}\\
	&\subseteq \Cl(G) \otimes \Cl(H),
	\end{align*}
so that $G \times H$ has the ITAP.
\end{proof}

\begin{cor} Let $G$ and $H$ be countable discrete groups. Suppose $G$ has the AP. Then $G \times H$ has the ITAP if and only if $H$ has the ITAP. 
\qed
\end{cor}

\section{The Fubini Crossed Product}\label{section crossed}

In this section, we study the crossed product version of the Fubini product.
\subsection{The Fubini Crossed Product}

Let $G$ be a countable discrete group. A {\em $G$-operator space} is an operator space equipped with a completely isometric action of $G$. 

Let $A \subseteq \B(\H)$ be an operator space equipped with a completely isometric action $\alpha$ of $G$. Define a new action $\pi$ of $A$ on $\H \otimes l^{2}G$ by $\pi(a) (v \otimes \delta_{g}) \coloneqq \alpha_{g^{-1}}(a)v \otimes \delta_{g}$ and let $G$ act on $\H \otimes l^{2}G$ by $\lambda(g) (v \otimes \delta_{h}) \coloneqq v \otimes \delta_{gh}$.

The {\em reduced crossed product} $A \rrtimes G$ is defined as the operator space spanned by $\{ \pi(a) \lambda(g) \in \B(\H \otimes l^{2}(G)) \mid a \in A, g \in G\}$.

\begin{defn}[({\cite[Lemma 2.1]{MR1721796}})] For any $\psi \in \B(l^{2} G)_{*}$, there is a natural, completely bounded  slice-map $\id_{A} \rrtimes \psi \colon A \rrtimes G \to A$. If $S \subseteq A$ is a $G$-invariant operator subspace, then the {\em Fubini crossed product} $F(S, A \rrtimes G)$ is defined as the set of all $x \in A \rrtimes G$ such that $(\id_{A} \rrtimes \psi)(x) \in S$ for all $\psi \in \B(l^{2}G)_{*}$. 
\end{defn}

The slice map $\id_{A} \rrtimes \psi $ is given by the restriction of the  von Neumann slice map which maps $A^{**} \overline{\otimes} \B(l^2 G)$ to $A^{**}$.
\begin{rem} We  have 
	\begin{equation*}
	S \rrtimes G \subseteq F(S, A \rrtimes G) \subseteq A \rrtimes G.
	\end{equation*}
In fact, many of the formal properties of the Fubini product hold for Fubini crossed products, usually with the same proof.
\end{rem}

\begin{defn}
We say that $(S, A \rrtimes G)$ has the {\em slice map property} if 
	\begin{equation*}
	F(S, A \rrtimes G) = S \rrtimes G.
	\end{equation*}
\end{defn}

\begin{lem}\label{lem trivial action} If the action of $G$ on $A$ is trivial, then 
	\begin{equation*}
	F(S, A \rrtimes G) = F(S, \Cl(G), A \otimes \Cl(G)).
	\end{equation*}
\end{lem}
\begin{proof}
Since the action of $G$ on $A$ is trivial, we have $A \rrtimes G  = A \otimes \Cl(G)$. The inclusion $\xymatrix{\Cl(G) \ar@{^{(}->}[r] & \B(l^{2}G)}$ gives the diagram 
	\[\xymatrix{
	\B(l^{2}G)^{*} \ar@{->>}[r]& \Cl(G)^{*}\\
	\B(l^{2}G)_{*} \ar@{^{(}->}[u] \ar@{-->}[ru]& 
	}.\] 
Hence we have $F(S, \Cl(G), A \otimes \Cl(G)) \subseteq F(S, A \rrtimes G)$.
By Goldstine's theorem for any functional $\psi \in \Cl(G)^{*}$, there exists a bounded net $\psi_{n} \in B(l^{2}(G))_{*}$ with $\|\psi_{n}\| \le \|\psi\|$, converging to $\psi$ in the weak$^{*}$ topology. Now it is easy to see that for any $x \in A \otimes \Cl(G)$, the elements $\id_{A} \otimes \psi_{n}(x)$ converge to $\id_{A} \rtimes \psi(x)$ in norm. It follows that $F(S, A \rrtimes G) \subseteq F(S, \Cl(G), A \otimes \Cl(G))$.

%
\end{proof}

\begin{ex} Let $l^{\infty}(G)$ act on $l^{2}G$ by multiplication. Then $G$ acts on $l^{\infty}(G)$ by left multiplication and $c_{0}(G) \subseteq l^{\infty}(G)$ is an invariant \cast-subalgebra. The Fubini crossed product $F(c_{0}(G), l^{\infty}(G) \rrtimes G)$ is the ideal of all ghost operators on $|G|$. Thus $G$ is exact if and only if $(c_{0}(G), l^{\infty}(G) \rrtimes G)$ has the slice map property by \cite{MR3146831}.
\end{ex}

\begin{lem}\label{lem int G} For families of $G$-invariant operator subspaces $\{S_{\alpha} \subseteq A\}$ we have
	\begin{equation*}
	F(\cap_{\alpha} S_{\alpha}, A \rrtimes G) = \cap_{\alpha} F(S_{\alpha}, A \rrtimes G).
	\end{equation*}	
\qed
\end{lem}

\begin{prop}[({cf.\ \cite[Proposition 2.2]{MR1721796}})]\label{prop ker G} Let $B$ and $D$ be $G$-operator spaces and let $\sigma\colon B \to D$ be a completely bounded $G$-equivariant map. Then
	\begin{equation*}
	F(\ker(\sigma), B \rrtimes G) = \ker (\sigma \rrtimes G).
	\end{equation*}
\qed
\end{prop}
%

\begin{thm} Let $G$ and $H$ be countable discrete groups and let $A$ be a $(G\times H)$-operator space. Then
	\begin{equation*}
	F(A^{H}, A \rrtimes G) = (A \rrtimes G)^{H}.
	\end{equation*}
\end{thm}
\begin{proof}
For $h \in H$, let $\alpha_{h} \colon A \to A$ denote the action by $h$. Then $\alpha_{h}$ is $G$-equivariant. Let $\sigma_{h} \coloneqq \id_{A} - \alpha_{h}$. Then $\sigma_{h}  \rrtimes G = \id_{A \rrtimes G} - \alpha_{h} \rrtimes G \colon A \rrtimes G \to A \rrtimes G$. The proof follows from Proposition~\ref{prop ker G} and Lemma~\ref{lem int G}, since
	\begin{align*}
	\cap_{h \in H} \ker(\sigma_{h}) &= A^{H},\\
	\cap_{h \in H} \ker(\sigma_{h} \rrtimes G) &= (A \rrtimes G)^{H}.
	\end{align*}
\end{proof}

\begin{ex}\label{ex fubini cross of c} Let $l^{\infty}(G)$ act on $l^{2}G$ by multiplication. Then $G$ acts on $l^{\infty}(G)$ by left  multiplication. As already pointed out $\Cu |G|  \cong l^{\infty}(G) \rrtimes G$; moreover, thinking of $\C$ as embedded into $l^{\infty}(G)$ via the constant functions we have $\Cu(|G|)^{G} = F(\C, l^{\infty}(G) \rrtimes G)$. Thus $G$ has the ITAP if and only if $(\C, l^{\infty}(G) \rrtimes G)$ has the slice map property.
\end{ex}

\begin{prop}\label{prop AP to SMP}
Let $G$ be a discrete group with the AP.  Then for any $S \subseteq A$, we have $F(S,A  \rrtimes G) = S \rrtimes G$.
\end{prop}
\begin{proof}
Since $G$ has the AP there exists a net $(u_{\alpha})$ in $M_{0}A(G) \cap c_{c}(G)$ converging to $1 \in M_{0}A(G)$ in the $\sigma (M_0A(G), Q(G))$ topology. As explained in \cite{MR2215118}, this implies that the net of Schur multipliers $\hat{M}_{u_{\alpha}} \in CB(\Cu(|G|,A))$ given by $\hat{M}_{u_{\alpha}}([a_{s,t}]) = [u_{\alpha} (st^{-1}) a_{s,t}]$ converges to the identity map in the point norm topology. Fixing a faithful representation $ A \hookrightarrow B(K)$ we can think of $A \rrtimes G \subseteq B(K \otimes \ell^2(G)) $ as matrices indexed by $G$ with entries in $A$. Since the $\langle \delta_s, \cdot \; \delta_t \rangle$ is a normal functional on $B(\ell^2(G))$ whose slice map gives the $s,t$ entry of the matrix in $B(K \otimes \ell^2(G)) $ we can characterise $F(S, A \rrtimes G)$ as those matrices in $A \rrtimes G$ with entries in $S$. Thus it is clear that $F(S, A \rrtimes G) \subseteq \Cu(|G|, A)$. Moreover, since each $u_{\alpha}$ has finite support, it is easy to check that $\hat{M}_{u_{\alpha}} (F(S, A \rrtimes G))$ is contained in  $\textup{span} \{ \pi (s) \lambda (g) \mid s \in S,~ g \in G \} \subseteq S \rrtimes G$. It follows that $S \rrtimes G \ni \hat{M}_{\alpha}(x) \to  x $ for any $x \in F(S, A \rtimes G)$, concluding the proof.
\end{proof}

\begin{prop} If $F(S, A \rrtimes G) = S \rrtimes G$ for all $S \subseteq A$, then $G$ has the AP.
\end{prop}
\begin{proof} Considering trivial actions we see that $\Cl(G)$ has the slice map property for any operator space $A$ by Lemma \ref{lem trivial action}, thus have the OAP by Kraus' theorem (Theorem \ref{thm kraus}). Hence $G$ has the AP by Haagerup-Kraus theorem (Theorem \ref{thm HK}).
\end{proof}

\subsection{Functoriality}

\begin{lem} Let $S \subseteq A$ and $T \subseteq B$ be $G$-operator spaces. Suppose that a completely bounded map $\sigma\colon A \to B$ maps $\sigma(S) \subseteq T$.
Then
	\begin{equation*}
	(\sigma \rrtimes G)(F(S, A \rrtimes G)) \subseteq F(T, B \rrtimes G).
	\end{equation*}
\end{lem}
\begin{proof} Let $x \in F(S, A \rrtimes G)$. For any $\psi \in \B(l^{2} G)_{*}$, we have $(\id_{A} \rrtimes \psi)(x) \in S$, thus
	\begin{align*}
	(\id_{A} \rrtimes \psi)[(\sigma \rrtimes G) (x)] 
	&= \sigma [(\id_{A} \rrtimes \psi)(x)]
	\end{align*}
belongs to $T$. Hence the lemma holds.
\end{proof}

\begin{lem} Let $S \subseteq A$ be $G$-operator spaces. Let $\sigma \colon A \to A$ be a completely bounded $G$-map. If $\sigma$ restricts to the identity on $S$, then $\sigma \rrtimes G$ restricts to the identity on $F(S, A \rrtimes G)$.
\end{lem}
\begin{proof} For $x \in F(S, A \rrtimes G)$ and $\psi \in \B(l^{2}G)_{*}$, we have
	\begin{align*}
	(\id_{A} \rrtimes \psi)[(\sigma \rrtimes G)(x)] = \sigma[(\id_{A} \rrtimes \psi)(x)] =(\id_{A} \rrtimes \psi)(x).
	\end{align*}
Thus $(\sigma \rrtimes G)(x) = x$.	
\end{proof}

\begin{cor}[({cf.~\cite[Lemma 2]{MR567831}})] Let $S \subseteq A$ and $S \subseteq B$ be $G$-operator spaces. If $\phi \colon A \to B$ and $\psi \colon B \to A$ are completely bounded $G$-maps such that $(\psi \circ \phi)_{|S} = \id_{S}$ and $(\phi \circ \psi)_{|S} = \id_{S}$, then there is a completely bounded $G$-isomorphism $F(S, A \rrtimes G) \to F(S, B \rrtimes G)$ which is the identity on $S \rrtimes G$.  
\qed
\end{cor}



Injective envelopes of $C^*$-algebras, operator systems and operator spaces have been considered by various authors \cite{MR519044,MR814075,MR1863404,MR2016568}. The $G$-injective envelope has only been defined for operator systems in the literature. However, the definitions and constructions are all analogous. First one has to find an injective extension of the given object in the appropriate category and then minimise it in such a way that uniqueness is automatic. Therefore in the following, we omit the proofs. For a $G$-operator space $S$, we denote the $G$-injective envelope by $I_{G}(S)$. We write $I(S)$ if $G$ is trivial.

\begin{lem}[({cf. \cite[Theorem 4]{MR567831}})]\label{lem univ fubini cross} Let $S$ be a $G$-operator system and let $A$ be a $G$-injective operator system containing $S$. Then the Fubini crossed product $F(S, A \rrtimes G)$ is independent of $A$. \qed
\end{lem}

\begin{defn}
The Fubini crossed product $F(S, A \rrtimes G)$ is called the {\em universal Fubini crossed product} of $S$ by $G$ and denoted $F(S, G)$. It is the largest Fubini crossed product of $S$.

We say that $S$ has the {\em universal slice map property} for $G$ if $S \rrtimes G = F(S, G)$.
\end{defn}


\begin{ex} For a discrete group $G$, the space $c_{0}(G)$ has the universal slice map property if and only if $G$ is exact and the space $\C$ has the universal slice map property if and only if $G$ has the ITAP.
\end{ex}

\begin{thm} Let $G$ be a discrete group. Then $(\Cu|G|)^{G} \subseteq I(\Cl(G))$. \qed
\end{thm}
\begin{proof} We only sketch the proof.

Since $l^{\infty}(G)$ is $G$-injective, we see that $F(\C, G) \cong F(\C, l^{\infty}(G) \rrtimes G) \cong (\Cu|G|)^{G}$ by Lemma~\ref{lem univ fubini cross} and Example~\ref{ex fubini cross of c}. On the other hand, for any $S$ we have $S \rrtimes G \subseteq F(S, G) \subseteq I_{G}(S) \rrtimes G \subseteq I(S \rrtimes G)$.
\end{proof}

%
%
%
%
%
%
%
%
%
%


\end{document}